\documentclass[preprint,12pt]{amsart}

\usepackage{amsmath,amssymb}

\def\End{\mathop {\fam 0 End}\nolimits}
\def\Ker{\mathop {\fam 0 Ker}\nolimits}
\def\tr{\mathop {\fam 0 tr}\nolimits}
\def\idd{\mathop {\fam 0 id}\nolimits}

\def\dbl{\lbrace\kern-3pt\lbrace}
\def\dbr{\rbrace\kern-3pt\rbrace}
\def\dblb{[\kern-2pt[}
\def\dbrb{]\kern-2pt]}

\newtheorem{lem}{Lemma}
\newtheorem{cor}{Corollary}
\newtheorem{prop}{Proposition}
\newtheorem{thm}{Theorem}

\theoremstyle{definition}
\newtheorem{defn}{Definition}
\newtheorem{exmp}{Example}
\newtheorem{rem}{Remark}

\title{Simple finite-dimensional double algebras}
\author{M.~E. Goncharov, P.~S. Kolesnikov}
\address{Sobolev Institute of Mathematics, Novosibirsk, Russia}
\thanks{Supported by Russian Science Foundation (project 14-21-00065)}

\begin{document}

\begin{abstract}
A double algebra is a linear space $V$ equipped with linear map $V\otimes V\to V\otimes V$. 
Additional conditions on this map lead to the notions of Lie and associative double algebras. 
We prove that simple finite-dimensional Lie double algebras do not exist over an arbitrary field,
and all simple finite-dimensional associative double algebras over an algebraically closed field 
are trivial. 
Over an arbitrary field, every simple finite-dimensional associative double algebra is commutative.
A double algebra structure on a finite-dimensional space $V$ 
is naturally described by a linear operator $R$ on the algebra $\End V$ of linear transformations of~$V$. 
Double Lie algebras correspond in this sense to 
skew-symmetric Rota---Baxter operators, double associative algebra structures---to (left) averaging operators.
\end{abstract}

\maketitle

\section{Introduction}

The general philosophy of noncommutative geometry which goes back to M. Kontsevich
states that a noncommutative geometric structure on an associative algebra $A$ should turn 
into an ordinary geometric structure on the variety of $n$-dimensional representations of $A$
under the functor $\mathrm{Rep}_n$ from the category of associative algebras to the category 
of schemes. 
In particular, the notion of a double Poisson algebra introduced in \cite{CBEG2007} and \cite{VdB2008}
fits this ideology (however, this approach is different from the one in \cite{Kon93}). 

Namely, suppose $A$ is a finitely generated associative algebra over a field $\Bbbk $, $n\ge 1$, and let 
$\mathcal O(\mathrm{Rep}_n(A))$ be the algebra of regular functions on the variety of all $n$-dimensional 
representations of $A$. This affine algebra is generated by functions $x^a_{ij}$, $a\in A$, $i,j=1,\dots, n$, where 
$\rho(a)_{ij} = x_{ij}^a(\rho)$ for every representation $\rho: A\to M_n(\Bbbk )$. 
A {\em double Poisson bracket} on $A$ is a linear map 
\[
 \dblb \cdot , \cdot \dbrb : A\otimes A \to A\otimes A
\]
satisfying a series of identities similar to anti-commutativity, Jacobi identity, and Leibniz rule.
A double Poisson bracket on $A$ induces ordinary Poisson bracket on $\mathcal O(\mathrm{Rep}_n(A))$
by the rule $\{x_{ij}^a,x_{kl}^b\} = \sum x_{il}^{c_{(1)}}x_{kj}^{c_{(2)}}$ for 
$\dblb a,b\dbrb = \sum c_{(1)}\otimes c_{(2)}\in A\otimes A$, $a,b\in A$ 
(it is enough to define Poisson bracket 
on the generators).

The ``double analogues'' of anti-commutativity and Jacobi identity involve only double bracket 
and do not involve the product in $A$. It was proposed in \cite{DSKV2015} to 
define {\em Lie double algebras} as linear spaces equipped with such double brackets.

This work was inspired by a problem stated by Victor Kac in his talk on the conference ``Lie and Jordan algebras, 
their representations and applications'' 
dedicated to Efim Zelmanov's 60th birthday (Bento Goncalves, Brasil, December 2015): 
prove that simple finite-dimensional Lie double algebras do not exist. 
In this paper, we present a solution of this problem. We also define associative double algebras
in such a way that its double commutator algebra is a Lie one. It turns out that, over an 
algebraically closed field, the only simple finite-dimensional double associative algebras are 1-dimensional;
over an arbitrary field, such system may exist but they are all commutative.

\section{Double associative and Lie algebras}

A {\em double algebra} is a linear space $V$ equipped with a linear map (called {\em double bracket})
$\dbl \cdot ,\cdot \dbr : V\otimes V\to V\otimes V$. 
It is clear how to define subalgebras and homomorphisms 
of double algebras. Ideals of a double algebra are supposed to be kernels of homomorphisms, 
so they have to be subspaces $I\subseteq V$ such that $\dbl V,I\dbr + \dbl I,V\dbr \subseteq I\otimes V + V\otimes I$. 
A double algebra $V$ is said to be simple if $\dbl V,V\dbr \ne 0$ and there are no nonzero proper ideals in $V$.

There is a natural way to extend a double bracket on $V$ to the following four linear maps (see \cite{DSKV2015})
$V^{\otimes 3}\to V^{\otimes 3}$: 
\[
\begin{gathered}
a\otimes b\otimes c \mapsto  \dbl a, b\otimes c \dbr _L = \dbl a,b \dbr \otimes c, \\
a\otimes b\otimes c \mapsto  \dbl a, b\otimes c\dbr _R = (b\otimes \dbl a,c\dbr )^{(12)}, \\
a\otimes b\otimes c \mapsto  \dbl a\otimes b, c\dbr _L = (\dbl a,c\dbr \otimes b)^{(23)}, \\
a\otimes b\otimes c \mapsto  \dbl a\otimes b, c \dbr _R = a\otimes \dbl b,c\dbr 
\end{gathered}
\]
for $a,b,c\in V$. Hereanafter $u^\sigma $ for $u\in V^{\otimes n}$, $\sigma \in S_n$ stands 
for the permutation of tensor factors.

\begin{defn}[\cite{DSKV2015}]\label{defn:dLie}
 A double algebra $L$ is said to be a {\em Lie} one if 
 \[
  \dbl a,b\dbr =- \dbl b,a\dbr ^{(12)}, \quad \dbl a, \dbl b,c\dbr \dbr _L -\dbl b, \dbl a,c\dbr \dbr _R^{(12)} = \dbl \dbl a,b\dbr,c\dbr _L
 \]
for all $a,b,c\in L$.
\end{defn}

\begin{defn}\label{defn:dAss}
 A double algebra $V$ is said to be {\em associative} if 
 \[
  \dbl a, \dbl b,c\dbr \dbr _L = \dbl \dbl a,b\dbr ,c\dbr _L, \quad \dbl a,\dbl b,c\dbr \dbr _R = \dbl \dbl a,b\dbr ,c\dbr _R
 \]
for all $a,b,c \in V$.

If, in addition, $\dbl a,b \dbr= \dbl b,a\dbr ^{(12)}$ then $V$ is a {\em commutative double algebra}.
\end{defn}

\begin{rem}
 Our definition for $\dbl a, b\otimes c\dbr _R$ is slightly different from analogous one in \cite{DSKV2015}, 
 but we also change the ``double analogue'' of Jacobi identity to get the same notion of a double Lie algebra as in \cite{DSKV2015}.
\end{rem}

\begin{exmp}
For a linear space $V$, define $\dbl u,v\dbr = u\otimes v$ for $u,v\in V$.
This turns $V$ into an associative and commutative double algebra denoted $V_c$.
\end{exmp}

\begin{exmp}
Given a double algebra $V$, define new (opposite) double bracket by
$\dbl u,v\dbr^{op}=\dbl u,v\dbr^{(12)}$, $u,v\in V$. Denote the double algebra obtained by $V^{op}$. 
If $V$ is associative (Lie, commutative) then so is $V^{op}$.
\end{exmp}

\begin{exmp}\label{exmp:2dimNonCom}
Let $V = \Bbbk ^2$ with standard basis $e_1$, $e_2$. Define $\dbl e_1,e_1\dbr = e_1\otimes e_2$, 
and let the other double products be zero. The double algebra $V_2$ obtained is associative and non-commutative.
\end{exmp}

\begin{exmp}
 Let $V$ be a linear space. Choose a linear map $\varphi \in \End V$ such that $\varphi^2=0$ and define
 \[
  \dbl u,v \dbr = \varphi(u)\otimes v + u\otimes \varphi(v), \quad u,v\in V.
 \]
The system obtained is an associative and commutative double algebra.
\end{exmp}

\begin{exmp}
 Let $A$ be an associative algebra, and let $Z$ be a subspace of $\{\varphi \in \End A \mid \varphi(xy)=x\varphi(y)\, x,y\in A\}$
 with a fixed basis $\varphi_1,\dots, \varphi_n$.
 Consider $V=Z\oplus A$ equipped with double bracket 
\[
 \begin{gathered}
  \dbl x,y\dbr = \sum\limits_{i=1}^n  \varphi_i(x)y\otimes\varphi_i, \quad x,y\in A, \\
  \dbl x,\varphi \dbr = \dbl \varphi, x\dbr = 0,\quad x\in V,\ \varphi\in Z.
 \end{gathered}
\]
Then $V$ is an associative, noncommutative double algebra. 
Example \ref{exmp:2dimNonCom} is a particular case of such $V$ with $A=\Bbbk $.
\end{exmp}

\begin{lem}\label{lem:dCommutator}
 If $V$ is a associative double  algebra with double bracket $\dbl \cdot , \cdot\dbr $. 
 Then the same space $V$ equipped with double bracket 
 \[
  \dblb a,b\dbrb  = \dbl a,b\dbr - \dbl b,a \dbr ^{(12)}, \quad a,b\in V,
 \]
is a Lie double  algebra denoted by $V^{(-)}$.
\end{lem}

\begin{proof}
 This may be shown in a straightforward computation. In finite-dimensional case, this statement independently follows 
 from a relation between averaging and Rota---Baxter operators on $\End V$ which will be considered below.
\end{proof}

Let $V$ and $U$ be two double algebras. 
Then $V\otimes U$ is also a double algebra with respect to a double bracket given by 
\[
 \dbl v_1\otimes u_1 , v_2\otimes u_2\dbr  = (\dbl v_1,v_2\dbr \otimes \dbl u_1,u_2\dbr )^{(23)} \in V\otimes U\otimes V\otimes U.
\]

Let $V$ be a Lie (associative, commutative) double algebra, and let $U$ be a commutative double  algebra.
Then $V\otimes U$ is a double Lie (resp., associative, commutative) algebra.
Indeed, it is enough to compute $\dbl \dbl a,b\dbr ,c \dbr _{L,R}$ and 
$\dbl a,\dbl b,c\dbr\dbr _{L,R}$ for $a,b,c\in V\otimes U$. 

If $V$ is a finite-dimensional double algebra then the conjugate map $\dbl \cdot , \cdot\dbr^*$
determined a double algebra structure on the dual space $V^*$. If $V$ is Lie or commutative double algebra 
then so is $V^*$, but for an associative double algebra $V$ this is not true.

\begin{exmp}
Consider the space $\Bbbk^2$ equipped with a double product $\dblb e_1,e_1\dbrb = e_1\otimes e_2-e_2\otimes e_1$
(others are zero). This is a Lie double algebra $L_2$ isomorphic to $V_2^{(-)}$.

In contrast, $L^*_2=\Bbbk^2$ with $\dblb e_1,e_2\dbrb = e_1\otimes e_1 = -\dblb e_2,e_1\dbrb$
is also a Lie double algebra, but it cannot be presented as $V^{(-)}$ for an associative double algebra~$V$.
\end{exmp}

\begin{exmp}[\cite{VdB2008}]
The space $P_1=\Bbbk[t]$ equipped with 
\[
\dblb t^n,t^m\dbrb = \frac{(t^n\otimes 1-1\otimes t^n)(t^m\otimes 1-1\otimes t^m)}{t\otimes 1-1\otimes t}
\]
is a Lie double algebra.
\end{exmp}

\begin{exmp}[Communicated by Victor Kac]
Consider the Lie double  algebra
 $dY(N)= P_1\otimes V_c^{op}\otimes V_c$, 
 where 
 $V=\Bbbk^ N$.
Its multiplication table relative to the basis
$T_n^{ij} =t^n\otimes e_i\otimes e_j$ ($n\ge 0$, $i,j=1,\dots, N$) has the following form:
\[
\dblb T_m^{ij}, T_n^{kl}\dbrb =
\sum\limits_{r=0}^{\min \{m,n\}-1} \big (T_r^{kj}\otimes T^{il}_{m+n-r-1}
 - T^{kj}_{m+n-r-1}\otimes T^{il}_r \big )
\]
It worths mentioning that these relations verbally repeat the 
defining relations of the Yangian $Y(gl_N)$:
\[
[ T_m^{ij}, T_n^{kl} ] =
\sum\limits_{r=0}^{\min \{m,n\}-1} \big (T_r^{kj} T^{il}_{m+n-r-1}
 - T^{kj}_{m+n-r-1} T^{il}_r \big )
\]
\end{exmp}

The relation between double Lie algebras and the classical Yang---Baxter equation has a very natural and precise 
form in the finite-dimensional case (c.f. \cite{ORS_2014}).
Suppose $V$ is a finite-dimensional space. Recall that the associative algebra $\End V$
of all linear transformations of $V$ has a symmetric bilinear nondegenerate form (trace form) 
\[
 \langle \cdot , \cdot \rangle : \End V\otimes \End V \to \Bbbk 
\]
given by $\langle x,y\rangle =\tr(xy)$, $x,y\in \End V$.
This form is invariant, i.e., $ \langle xy,z\rangle = \langle x,yz\rangle = \langle y,zx\rangle$.
Fix a linear isomorphism $\iota :\End V\to (\End V)^*$ given by
\[
 \langle \iota (x), y\rangle =\langle x,y\rangle 
\]
(here in the left-hand side $\langle \cdot ,\cdot \rangle $ denotes the natural pairing).

Recall that for every finite-dimensional space $W$ we may identify $\End W$ and $W^*\otimes W$
in the following way:
\[
 (\varphi \otimes x): y\mapsto \langle \varphi,z\rangle y,\quad x,y\in W,\ \varphi\in W^*. 
\]
The latter allows to identify $\End V\otimes \End V$ and $\End(\End V)$ by means of the trace form.
Thus we have a chain of isomorphisms
\begin{multline}\label{eq:OP_iso}
 \End(V\otimes V) \simeq (V\otimes V)^*\otimes V\otimes V \simeq V^{*}\otimes V \otimes V^*\otimes V \\
  \simeq \End V\otimes \End V\simeq \End(\End V). 
\end{multline}
Therefore, the space of double brackets on $V$ is isomorphic to the space $\End(\End V)$, i.e., every double algebra structure $\dbl \cdot,\cdot\dbr \in \End(V\otimes V)$ 
is determined by a linear operator $R:\End V\to \End V$. Tracking back the chain \eqref{eq:OP_iso} we obtain an 
explicit expression for a double bracket in terms of operators:
\begin{equation}\label{eq:Bracket_via_OP}
 \dbl a,b\dbr  = \sum\limits_{i=1}^N e_i(a)\otimes R(e_i^*)(b) = \sum\limits_{i=1}^N R^*(e_i)(a)\otimes e^*_i(b), \quad a,b\in V, 
\end{equation}
where $e_1,\dots, e_N$ is a linear basis of $\End V$, $e_1^*,\dots, e_N^*$ is the corresponding dual basis 
relative to the trace form,  $R^*$ denotes the conjugate operator on $\End V$ relative to the trace form.

\begin{lem}\label{lem:op_Equivalence}
Let $A$ be a finite-dimensional algebra (not necessarily associative) 
equipped with a symmetric bilinear invariant nondegenerate form $\langle \cdot,\cdot \rangle $.
Suppose $R$ is a linear operator on $A$ and $R^*$ stands for its conjugate.
Then either of the following three identities is equivalent to other two:
\begin{gather}
 R(x)R(y) = R(R(x)y),\label{eq:lAve-1} \\
 R^*(x)R(y) = R^*(xR(y)), \label{eq:lAve-2} \\
 R^*(R(x)y) = R^*(xR^*(y)). \label{eq:lAve-3}
\end{gather}
Moreover, the same holds for the following triple of identities:
\begin{gather}
 R^*(x)R^*(y) = R^*(R^*(x)y), \label{eq:lAve-1d} \\
 R(x)R^*(y) = R(xR^*(y)), \label{eq:lAve-2d} \\
 R(R^*(x)y) = R(xR(y)). \label{eq:lAve-3d}
\end{gather}
\end{lem}

\begin{proof}
Note that for every $x,y,z\in A$ we have 
\[
 \begin{gathered}
 \langle R(y)R(x),z\rangle = \langle x, R^*(zR(y))\rangle = \langle y, R^*(R(x)z)\rangle , \\
 \langle R(R(y)x), z\rangle = \langle x, R^*(z)R(y) \rangle = \langle y, R^*(xR^*(z))\rangle . \\
 \langle R(yR(x)),z\rangle = \langle x, R^*(R^*(z)y)\rangle = \langle y, R(x)R^*(z)\rangle , \\
 \langle R(R^*(y)x),z\rangle = \langle x, R^*(z)R^*(y)\rangle = \langle y, R(xR^*(z))\rangle , \\
\end{gathered}
\]
Nondegeneracy of $\langle \cdot,\cdot \rangle$ implies the claim.
\end{proof}

\begin{thm}\label{thm:OP_RB_ave}
 Let $V$ be a double algebra with a double bracket $\dbl \cdot,\cdot \dbr $ determined by an operator $R:\End V\to \End V$ by \eqref{eq:Bracket_via_OP}.
 Then
\begin{enumerate}
\item $V$ is a Lie double algebra if and only if 
\begin{equation}\label{eq:RB_skew}
  R^*=-R,\quad R(x)R(y) = R(R(x)y)+R(xR(y)), \quad x,y\in \End V.
\end{equation}
\item 
 $V$ is an associative double algebra if and only if 
\begin{equation}\label{eq:Ave_dual}
 R(x)R(y) = R(R(x)y),\quad R^*(x)R^*(y) = R^*(R^*(x)y), \quad    x,y\in \End V.
\end{equation}
\item 
 $V$ is a commutative double algebra if and only if 
\begin{equation}\label{eq:Ave_herm}
 R=R^*,\quad R(x)R(y)=R(R(x)y)=R(xR(y)), \quad x,y\in \End V. 
\end{equation}
\end{enumerate}
\end{thm}

\begin{proof}
Equation \eqref{eq:Bracket_via_OP} immediately implies the identity
$\dbl a,b\dbr =\pm \dbl b,a\dbr ^{(12)}$ to be equivalent to $R=\pm R^*$.

Suppose $F_{12}\in (\End V)^{\otimes 3}\simeq \End (V^{\otimes 3})$ is given 
by 
\[
 F_{12}(a\otimes b\otimes c) = \dbl a , \dbl b,c\dbr \dbr _L
 =\sum\limits_{i,j=1}^N e_j(a)\otimes R(e_j^*)e_i(b) \otimes R(e_i^*)(c),  
 \quad a,b,c\in V.
\]
For every $x,y\in \End V$, compute 
\begin{multline}\label{eq:F12}
 (\langle x,\cdot \rangle \otimes \langle y,\cdot \rangle \otimes \idd)F_{12}
 = \sum\limits_{i,j=1}^N \langle x,e_j\rangle \langle y,R(e_j^*)e_i \rangle R(e_i^*) \\
 =\sum\limits_{i=1}^N\left\langle y, \sum\limits_{j=1}^N\langle x,e_j\rangle R(e^*_j)e_i\right\rangle R(e_i^*)
 =\sum\limits_{i=1}^N \langle y,R(x)e_i\rangle R(e_i^*) \\
 =\sum\limits_{i=1}^N \langle yR(x), e_i\rangle R(e_i^*)
 =R(yR(x)).
\end{multline}

Similarly, if 
\[
 F_{23}(a\otimes b\otimes c) = \dbl b, \dbl a,c\dbr \dbr _R^{(12)} = 
 \sum\limits_{i,j=1}^N e_j(a)\otimes e_i(b) \otimes R(e_i^*)R(e_j^*)(c)
\]
then
for every $x,y\in \End V$ we have 
\begin{equation}\label{eq:F23}
(\langle x,\cdot \rangle \otimes \langle y,\cdot \rangle \otimes \idd)F_{23} = R(y)R(x).
\end{equation}

On the other hand, if 
\[
\begin{gathered}
G_{12}(a\otimes b\otimes c) = \dbl \dbl a,b\dbr ,c \dbr _L =  \sum\limits_{i,j=1}^N e_ie_j(a)\otimes R(e_j^*)(b) \otimes R(e_i^*)(c),\\
G_{23}(a\otimes b\otimes c) = \dbl \dbl b,a\dbr ,c\dbr _R^{(12)} =  \sum\limits_{i,j=1}^N e_iR(e_j^*)(a)\otimes e_j(b)\otimes   R(e_i^*)(c)
\end{gathered}
\]
then
\begin{gather}
(\langle x,\cdot \rangle \otimes \langle y,\cdot \rangle \otimes \idd)G_{12} = R(R^*(y)x), \label{eq:G12} \\ 
(\langle x,\cdot \rangle \otimes \langle y,\cdot \rangle \otimes \idd)G_{23} = R(R(y)x). \label{eq:G23} 
\end{gather}

The first statement now follows from \eqref{eq:F12}, \eqref{eq:F23}, \eqref{eq:G12}. 
Relations \eqref{eq:F12}, \eqref{eq:G12} and  \eqref{eq:F23}, \eqref{eq:G23} imply the associativity of $V$ 
is equivalent to the following pair of identities:
\[
 R(yR(x))=R(R^*(y)x), \quad R(y)R(x) = R(R(y)x), \quad x,y\in \End V.
\]
To complete the proof it is enough to apply Lemma \ref{lem:op_Equivalence} to $A=\End V$.
\end{proof}

The second relation in \eqref{eq:RB_skew} is known as {\em Rota---Baxter equation}. 
Linear transformation of an associative algebra satisfying this equation is called a {\em Rota---Baxter operator}.
Associative algebras with Rota---Baxter operator (Rota---Baxter algebras) have a well-developed theory, see \cite{Guo_rbabook}.
Skew-symmetric Rota---Baxter operators are in one-to-one correspondence with constant 
solutions of the classical Yang---Baxter equation, see \cite{Semenov_83}.
The second relation in \eqref{eq:Ave_herm} is known as {\em averaging equation}. Algebras with 
such operators (averaging algebras)
are of substantial interest in functional analysis, 
they have also been studied from combinatorial point of view \cite{PeiGuo2015, JGuopRoz2015}. 

A linear map $R:A\to A$ on an algebra $A$ satisfying the first relation in \eqref{eq:Ave_dual}
is said to be a {\em left averaging} operator on~$A$.

\begin{cor}\label{cor:AveRB_op}
 Let $A$ be an algebra (not necessarily associative) with a symmetric bilinear invariant nondegenerate form. 
Suppose $T$ is a left averaging operator on $A$ such that its conjugate $T^*$ is also left averaging.
Then $R=T-T^*$ is a skew-symmetric Rota---Baxter operator on $A$.
\end{cor}

\begin{exmp}
A simple finite-dimensional Lie algebra $A$ over a field of zero characteristic 
satisfies the condition of Corollary \ref{cor:AveRB_op}. 
For a Lie algebra, left averaging operator is obviously an averaging one.
For example, for $A=sl(2,\mathbb C)$ all averaging operators (described in \cite{Kol2015_JMP}) 
are symmetric, so the only Rota---Baxter operator of the form $T-T^*$ is zero.
\end{exmp}

\section{Ideals in Lie double algebras}

Let us first state a necessary condition of simplicity of a double algebra.

\begin{prop}\label{prop:irreducible}
 Let $V$ be a simple finite-dimensional double algebra with a double product $\dbl \cdot ,\cdot \dbr$
corresponding to an operator $R:\End V\to \End V$.
Then $V$ has no nonzero proper invariant subspaces relative to all operators from $R(\End V)$ and $R^*(\End V)$.
\end{prop}

\begin{proof}
Relation \eqref{eq:Bracket_via_OP} immediately implies every   
$R(\End V)$- and $R^*(\End V)$-invariant subspace of $V$ to be an ideal.
\end{proof}

Recall that a subalgebra $B$ of $\End V$ is called irreducible if $V$ is an irreducible $B$-module.
The Jacobson Density Theorem (see, e.g., \cite{Jac}) implies an irreducible subalgebra $B$ to be isomorphic
to $\End_{\mathcal D^{op}} V$, where $\mathcal D$ is the centralizer of $B$ in $A$ (division algebra). 
In particular, if the base field $\Bbbk $ is algebraically closed then $B=\End V$ (Burnside Theorem). 
For an arbitrary field $\Bbbk $, an irreducible subalgebra $B$ has to contain the identity of $\End V$.

\begin{thm}\label{thm:NoLieSimple}
Let $L$ be a finite-dimensional Lie double algebra, $\dim L>1$. Then 
$L$ contains a nonzero proper ideal.
\end{thm}

\begin{proof}
 By Theorem \ref{thm:OP_RB_ave} the double bracket on $L$ is given by \eqref{eq:Bracket_via_OP} for 
 an appropriate skew-symmetric Rota---Baxter operator $R$ on the algebra $A=\End L$. 
 The Rota---Baxter relation 
 \[
  R(x)R(y) = R(R(x)y)+R(xR(y))
 \]
implies $R(A)$ to be a subalgebra of $A$. Note that $R(A)$ does not contain the identity $1\in A$. 
Indeed, if $R(x)=1$ for some $x\in A$ then $1\cdot 1 = R(x)+R(x)=1+1$ which is impossible.
This contradicts to Proposition \ref{prop:irreducible}.
\end{proof}

The only Lie double algebra $L$ without nonzero proper ideals is one-dimensional. In this case, $\dblb L,L\dbrb =0$, 
so simple finite-dimensional Lie double algebras do not exist. 

If $L$ is a finite-dimensional Lie double algebra then every ideal of $L$ may be embedded into a
maximal one which has to be of codimension one. On the other hand, every ideal of $L$ 
contains a minimal one. This observation causes natural question on the possible dimension of such minimal ideals. 

\begin{exmp}\label{exmp:L_2,n}
 Let $L(2,n)$ denote the Lie double algebra $L_2\otimes V_c\otimes V_c^{op}$, $V=\Bbbk^n$.
If $I$ is a nonzero ideal in $L(2,n)$ then $\dim I\ge n$.
\end{exmp}

\section{Ideals in double associative algebras}

Let $V$ be a linear space over a field $\Bbbk $ with double bracket $\dbl u,v \dbr =\alpha u\otimes v$, $\alpha \in \Bbbk $.
Then $V$ is a commutative double algebra and every subspace of $V$ is an ideal.
Thus this double algebra is simple if and only if $\dim V=1$, $\alpha \ne 0$. We are going to show that there are 
no other simple finite-dimensional associative double algebras. 

Throughout the rest of the section $V$ is a finite-dimensional associative double algebra of dimension $n$ 
with double bracket $\dbl \cdot ,\cdot \dbr $ 
given by \eqref{eq:Bracket_via_OP} for an appropriate operator $R$ on $A=\End V$ (equipped with the trace form)
satisfying \eqref{eq:lAve-1}--\eqref{eq:lAve-3d}.

\begin{lem}\label{lem:GeneralProp_of_R,R^*}
\begin{enumerate}
 \item 
    The space $B=R(A)+R^*(A)$ is a subalgebra of $A$.
 \item
    $R(A)$ and $R^*(A)$ are right ideals of $B$.
 \item
   $\Ker R$ and $\Ker R^*$ are left $B$-submodules of $A$.
 \item 
   $\Ker R \cdot R^*(A)\subseteq \Ker R\cap \Ker R^*$.
 \item 
   $\Ker R^* \cdot R(A) \subseteq \Ker R\cap \Ker R^*$.
\end{enumerate}
\end{lem}

\begin{proof}
immediately follow from \eqref{eq:lAve-1}--\eqref{eq:lAve-3d}.
\end{proof}

\begin{lem}\label{lem:NondegCase}
If $\Ker R=0$ then $R=R^*=\alpha \idd_A$ for $\alpha \ne 0$.
\end{lem}

\begin{proof}
 Since $\dim \Ker R=\dim\Ker R^*$, we have $R^*(A)=A$.
 Relations \eqref{eq:lAve-1} and \eqref{eq:lAve-2d} imply $R(xy)=xR(y)=R(x)y$ for all $x,y\in A$. 
 Hence, $R=\alpha \idd_A$ for a nonzero scalar~$\alpha$.
\end{proof}

\begin{prop}\label{prop:R+R^*=A}
 If $R(A)+R^*(A)=A$ then $\Ker R=0$.
\end{prop}

\begin{proof}
Assume $A=R(A)+R^*(A)$, $\Ker R\ne 0$.  
Then for every $P\in A$ there exists a decomposition $P=P_1+P_2$, $P_1\in R(A)$, $P_2\in R^*(A)$. 
Relations \eqref{eq:lAve-1} and \eqref{eq:lAve-2d} imply
\begin{equation}\label{eq:R_decompos}
  R(P) = R(P_1)+R(P_2) = R(P_1\cdot 1) + R(1\cdot P_2) = P_1R(1) + R(1)P_2.
\end{equation}

Then $R(A)$, $R^*(A)$, and $R(A)\cap R^*(A)$ are proper right ideals in $A$ and there exists a basis in $V$ such that 
in the corresponding matrix presentation 
 \[
  R(A)\cap R^*(A)=\left\{ \begin{pmatrix} x & y \\ 0 & 0 \end{pmatrix} \mid x\in M_{s}(\Bbbk ), y\in M_{s,n-s}(\Bbbk ) \right \}
 \]
for some $0\le s< n$. 

If $s>0$ then for every $u\in R(A)\cap R^*(A)$ 
\[
 R(u)=R(u\cdot 1) = uR(1), \quad R(u)=R(1\cdot u) = R(1)u
\]
by \eqref{eq:lAve-1}, \eqref{eq:lAve-2d}. 
Hence, 
$R(1)$ commutes with every matrix unit $e_{1i}$, $i=1,\dots, n$, 
so $R(1)=\alpha 1$ for some $\alpha \in \Bbbk $. Note that $\alpha \ne 0$: otherwise, 
$R(A)=0=R^*(A)$ by \eqref{eq:R_decompos}.
This implies $R(A)=A$ in contradiction with our assumption.

Therefore, $s=0$ and $R(A)\cap R^*(A)=0$, so $A=R(A)\dot + R^*(A)$, $n=2r$.
As a proper right ideal, 
\[
 R(A)=\left\{ \begin{pmatrix} x & y \\ 0 & 0 \end{pmatrix} \mid x\in M_{r}(\Bbbk ), y\in M_{r,n-r}(\Bbbk ) \right \}
\]
in an appropriate matrix presentation of $\End V$, $0<r<n$
Since $\Ker R^*=R(A)^\perp $ (relative to the trace form), we have 
\[
 \Ker R^* = \left\{ \begin{pmatrix} 0 & y \\ 0 & z \end{pmatrix} \mid y\in M_{r,n-r}(\Bbbk ), z\in M_{n-r}(\Bbbk ) \right \}
\]
As a complement of $R(A)$, the right ideal $R^*(A)$ is of the form
\[
\left\{ \begin{pmatrix} \psi(w) \\ w \end{pmatrix} \mid w\in M_{n-r,n}(\Bbbk ) \right \} 
\]
for a uniquely defined linear map $\psi: M_{n-r,n}(\Bbbk ) \to M_{r,n}(\Bbbk )$. Recall that $n-r=r$.
Obviously, 
\[
 \begin{pmatrix} \psi(w) \\ w \end{pmatrix} P = \begin{pmatrix} \psi(w)P \\ wP \end{pmatrix} = \begin{pmatrix} \psi(wP) \\ wP \end{pmatrix}
\]
for all $P\in M_n(\Bbbk )$, $w\in M_{n-r,n}(\Bbbk )$. Therefore, $\psi(w) = pw$ for some fixed $p\in M_r(\Bbbk )$, i.e.,
\[
R^*(A) =\left\{ \begin{pmatrix} px & py \\ x & y \end{pmatrix} \mid x\in M_r(\Bbbk), y\in M_{r,n-r}(\Bbbk ) \right \}.
\]
Finally, $\Ker R = R^*(A)^\perp $ (relative to the trace form), so 
\[
 \Ker R = \left\{ \begin{pmatrix} y &-yp \\ z& -zp  \end{pmatrix} \mid y\in M_{r,n-r}(\Bbbk ), z\in M_{n-r}(\Bbbk ) \right \}
\]

Suppose 
$R(1) = \begin{pmatrix} a & b \\ 0 & 0\end{pmatrix}$ for some $a,b\in M_r(\Bbbk )$.
Then by \eqref{eq:R_decompos} 
\[
 R \begin{pmatrix}
           0 & y \\ 0 & 0
          \end{pmatrix}
= \begin{pmatrix}
  0 & y \\ 0 & 0
  \end{pmatrix}
  \begin{pmatrix}
   a & b \\ 0 & 0
  \end{pmatrix} 
=0.
\]
Therefore, 
$\begin{pmatrix} 0 & y \\ 0 & 0 \end{pmatrix} \in \Ker R$ for all $y\in M_r(\Bbbk )$
in contradiction with Lemma~\ref{lem:GeneralProp_of_R,R^*}. 
\end{proof}

\begin{cor}
 Over an algebraically closed field $\Bbbk $, the only simple finite-dimensional associative double algebra 
 is a 1-dimensional space equipped with double product $\dbl u,v\dbr = \alpha u\otimes v$, $\alpha \in \Bbbk ^*$.
\end{cor}

\begin{proof}
 By Proposition \ref{prop:irreducible}, $B=R(A)+R^*(A)$ is an irreducible subalgebra of $A$. 
 The Burnside Theorem implies $B=A$, so by Proposition \ref{prop:R+R^*=A} and Lemma \ref{lem:NondegCase}
 $R=R^*=\alpha \idd_A$, $\alpha \in \Bbbk^*$. Therefore, 
 $\dbl u,v\dbr =\alpha v\otimes u$ for every $u,v\in V$. Such a double algebra is simple if and only if 
 $\dim V=1$ (so $u\otimes v=v\otimes u)$.
\end{proof}

Over an arbitrary field, simple finite-dimensional associative double algebras may exist, but they turn out to be commutative.  

\begin{lem}\label{lem:Ideal_in_Irred}
If $B$ is an irreducible subalgebra of $A$ and $I$ is a proper right ideal of $B$  then 
$IV\ne V$.
\end{lem}

\begin{proof}
 Assume $IV=V$. Since $B$ is irreducible, $V=Bv$ for every $0\ne v\in V$. Hence, 
 $V=IV=IBv\subseteq Iv$. Therefore, $I$ itself is an irreducible subalgebra of $A$ 
 and must contain the identity, so $I=B$.
\end{proof}

\begin{thm}\label{thm:R=R^*}
 Suppose $V$ is a simple finite-dimensional associative double algebra. Then $V$ is commutative. 
\end{thm}

\begin{proof}
As above, let the double algebra structure on $V$ be defined by an operator $R$ on the algebra $A=\End V$.

Consider $B=R(A)+R^*(A)$. By Proposition \ref{prop:irreducible} $B$ is an irreducible subalgebra of $A$. 
Lemma \ref{lem:GeneralProp_of_R,R^*} implies $R(A)$ and $R^*(A)$ to be right ideals of $B$. 

Assume $I=R(A)$ is a proper right ideal of $B$. Then by Lemma \ref{lem:Ideal_in_Irred} $IV\ne V$, and
\[
 \dbl u,v\dbr = \sum\limits_{i,j=1}^n e_{ij}(u)\otimes R(e_{ji})(v) \in V\otimes IV
\]
for all $u,v\in V$.
Therefore, $\dbl V,V\dbr \subset V\otimes IV$,
so $IV$ is a proper ideal of the double algebra $V$ (it is nonzero since $R=0$ otherwise).

Thus, $R(A)=B$ and  $R(A)=R^*(A)=B$. 
In this case, $R$ is a (two-sided) averaging operator on $A$, and it is easy to see that 
\[
J= \{R(x)-R^*(x) \mid x\in A \}
\]
is a (two-sided) ideal of $B$. However, $B$ is a simple algebra, so either $J=0$ or $J=B$. 
The latter case is impossible since $R-R^*$ is a Rota---Baxter operator by Corollary \ref{cor:AveRB_op}, 
but the image of a Rota---Baxter operator may not contain the identity of $A$ (so is $B$). 
Hence, $R=R^*$ and $V$ is commutative.
\end{proof}

\begin{rem}
There are two principal types of simple finite-dimensional double commutative algebras described
in terms of their corresponding symmetric averaging operators. It is easy to see from Lemma \ref{lem:GeneralProp_of_R,R^*}
that either $R(A)\cap \Ker R=0$ or $R(A)\cap \Ker R=R(A)$. 
In the first case, $A=R(A)\oplus \Ker R$, and for every $x\in A$ there exist uniquely defined $x_0\in \Ker R$ and $x_1\in R(A)$
such that $x=x_1+x_0$ and $R(x)=x_1u$, where $u=R(1)$ is a central element of $R(A)$. In particular, $R^2=uR$.
In the second case, $R^2=0$, in particular, $R(1)=0$ which is possible only if the characteristic of $\Bbbk $ divides $\dim V$.
Examples below show these two opportunities.
\end{rem}

\begin{exmp}
Let $\Bbbk =\mathbb R$, $N=2$. Consider the decomposition $M_2(\mathbb R) = E_1\oplus E_0$,
where 
\[
E_1=\left\{\begin{pmatrix}
                          x & y \\ -y & x
                         \end{pmatrix} \mid x,y\in \mathbb R \right\},
\quad 
E_0=
\left\{\begin{pmatrix}
                          x & y \\ y & -x
                         \end{pmatrix} \mid x,y\in \mathbb R \right\}.
\]
With $u=2$, this decomposition determines the following symmetric averaging operator:
\begin{equation}\label{eq:exmp_real}
 R\begin{pmatrix}
   x & y \\ v & w 
  \end{pmatrix} = 
\begin{pmatrix}
 x+w & y-v \\ v-y & x+w
\end{pmatrix}
\end{equation}
such that $R^2=2R$.
The corresponding commutative double algebra structure on $V=\mathbb R^2$ is given by the following multiplication table:
\[
\begin{gathered}
 \dbl e_1,e_1\dbr = -\dbl e_2,e_2\dbr = e_1\otimes e_1 -e_2\otimes e_2, \\
 \dbl e_1,e_2\dbr = e_1\otimes e_2 + e_2\otimes e_1.
\end{gathered}
\]
Let us show that $V$ is simple. Otherwise, there exists a 1-dimensional ideal spanned by $v=\alpha e_1+\beta e_2$, $\alpha,\beta\in \mathbb R$. 
By definition, the functional $\xi\otimes \xi\in (V\otimes V)^*$, $\xi = \beta e_1^* -\alpha e_2^*$, has to annihilate 
$\dbl v,e_1\dbr $ and $\dbl v,e_2\dbr $. This leads to a system of algebraic equations on $\alpha $ and $\beta $
which has only zero solution in $\mathbb R$. Hence, the double algebra determined by the operator \eqref{eq:exmp_real} is simple.
\end{exmp}

\begin{exmp}
Let $\Bbbk =\mathbb Z_2(t)$.
Then 
\[
 R\begin{pmatrix}
   x & y \\ v & w 
  \end{pmatrix} = 
\begin{pmatrix}
 x+w & y+tv \\ 
 t^{-1}y+v & x+w
\end{pmatrix}
\]
is a symmetric averaging operator on $M_2(\Bbbk )$ such that $R^2=0$. 
This operator determines a structure of a simple double commutative algebra on $\Bbbk^2$
(this may be shown in the same way as in the previous example).
\end{exmp}

\subsection*{Acknowledgements}
This work was supported by Russian Science Foundation (Project No. 14-21-00065). 
The authors are grateful to Victor Kac for stating the problem and inspiring discussions of the results.


\end{document}